\documentclass{amsart}
\usepackage{mathtools}
\usepackage{float}
\usepackage[utf8]{inputenc}   
\usepackage[T1]{fontenc}

\newtheorem{thm}{Theorem}
\newtheorem{prop}{Proposition}

\newtheorem{lem}{Lemma}

\newtheorem{definition}{Definition}[subsection]

\usepackage{hyperref}

\usepackage{lipsum}
\usepackage{mathtools}
\usepackage{float}

\usepackage{amssymb}
\usepackage{caption}
\usepackage{geometry}

\usepackage{quiver}
\usepackage{tikz-cd}

\newcommand{\opname}[1]{\operatorname{\textup{\textsf{#1}}}}

\newcommand{\size}{\opname{size}}

\newcommand{\lat}{\opname{lat}}
\newcommand{\Long}{\opname{long}}

\newcommand{\FF}{\opname}

\newcommand{\out}{\opname{out}}

\newcommand{\jaw}{\opname{jaw}}

\title{A bijection between rooted planar maps  and generalized fighting fish}

\author{Enrica Duchi}
\address{IRIF, Universit\'e Paris Cit\'e, France}
\email{duchi@irif.fr}

\author{Corentin Henriet}
\address{IRIF, Universit\'e Paris Cit\'e, France}
\email{henriet@irif.fr}

\begin{document}

	\begin{abstract}
		The class of fighting fish is a recently introduced model of branching surfaces generalizing parallelogram polyominoes. We can alternatively see them as gluings of cells, walks on the square lattice confined to the quadrant or shuffle of Dyck words. With these different points of view, we introduce a natural extension of fighting fish that we call \emph{generalized fighting fish}. We show that generalized fighting fish are exactly the Mullin codes of rooted planar maps endowed with their unique rightmost depth-first search spanning tree, also known as Lehman-Lenormand code. In particular, this correspondence gives a bijection between fighting fish and nonseparable rooted planar maps, enriching the garden of bijections between classes of objects enumerated by the sequence $\frac{2}{(n+1)(2n+1)} \binom{3n}{n}$. 
	\end{abstract}
	
	\maketitle

Fighting fish were first defined in 2016 by Duchi et al. in \cite{ff} as a model of branching surfaces that is a generalization of parallelogram polyominoes. While parallelogram polyominoes are enumerated according to their halfperimeter by the Catalan numbers, fighting fish are counted by the sequence $\frac{2}{(n+1)(2n+1)} \binom{3n}{n}, n \geq 1$. It is the sequence $A000139$ of the OEIS \cite{OEIS} and it counts some other combinatorial objects : nonseparable planar maps, synchronized intervals of the Tamari lattice, left ternary trees, two-stack sortable permutations. In the last few years, a growing number of articles has been dealing with the problem of connecting bijectively these different classes of objects in order to have a better combinatorial understanding of this family. The following diagram summarizes the currently known bijections, dashed (resp. plain) arrows indicating recursive (resp. direct) bijections.

\[\begin{tikzcd}[arrows={<->},row sep=30pt, column sep=80pt]
\substack{\text{nonseparable rooted}\\ \text{planar maps}}
\arrow[ rr,"\text{\cite{DDP}, \cite{JS}, \cite{Schaeffer-these}}"] \arrow[dr,color=red,"\text{this article}"] \arrow[dd,"\text{\cite{fangnonsep}}"] & & \text{left ternary trees}\\
 & \text{fighting fish} \arrow[dashed, dr,"\text{\cite{FangFish}}"] & \\
\substack{\text{synchronized intervals}\\ \text{of the Tamari lattice}} \arrow[ur,"\text{\cite{EFF}}"] & \substack{{}\\{}\\{}\\\text{\cite{GW96}, \cite{DGW96}}} & \substack{\text{two-stack sortable}\\ \text{permutations}} \arrow[dashed, rounded corners, to path={ -- ([yshift=-2ex]\tikztostart.south) -| ([xshift=-1.5ex]\tikztotarget.west) -- (\tikztotarget)}]{uull}
\end{tikzcd}\]

The present article is the second chapter of an extended version of a 12 page article accepted at the FPSAC 2022 conference \cite{fpsac22}, the first chapter being \cite{EFF} dealing with the bijective link between fighting fish and intervals of the Tamari lattice. The objective of this writing is twofold. The first purpose is to obtain a direct bijection between fighting fish and nonseparable planar maps. This bijection is in fact not new : it is the restriction of a construction proposed by Mullin 50 years ago encoding a planar map with a distinguished spanning tree by its contour word. However, it was not realized before that Mullin's bijection is mapping nonseparable planar maps (endowed with their rightmost depth-first search spanning tree) to fighting fish. In a second time, we define a natural generalization of fighting fish that we call generalized fighting fish, and we prove that they are exactly the Mullin encodings of rooted planar maps endowed with their rightmost DFS spanning tree also known as Lehman-Lenormand code \cite{Le70}, see also \cite[Chapter 2, II.3]{Co75}. It is important to remark that while the first article \cite{EFF} deals with \emph{extended} fighting fish, this one introduces \emph{generalized} fighting fish : these two classes of objects are two different extensions of the previously known class of fighting fish.

\section{Introduction and motivation}

\subsection{Fighting fish}

A complete presentation of fighting fish has been made in the introductive articles \cite{ff} and \cite{ff2}, and we restate here their definitions in the language of cell gluings used in these papers.
\begin{figure}[H]
	\centering
	\includegraphics[page=1,scale=0.6]{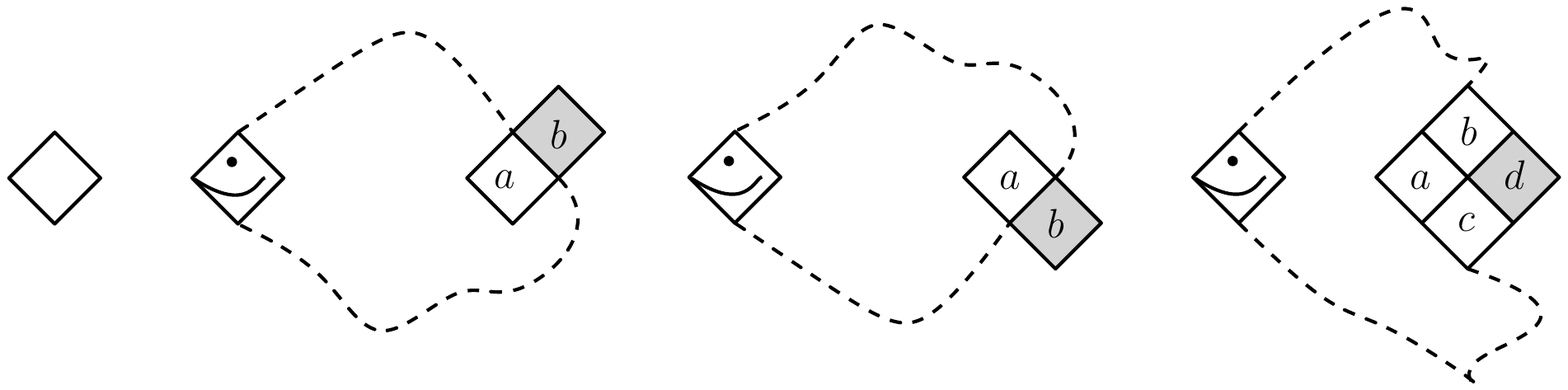}
	\caption{A cell and operations of upper, lower and double gluing.}
	\label{fig:gluings}
\end{figure}

A \emph{cell} is a 45 degree tilted unit square whose boundary is made out of four edges called \emph{left lower edge}, \emph{right lower edge}, \emph{right upper edge} and \emph{left upper edge}. We intend to build fighting fish as sets of cells glued together along some of their edges, so we define an edge of a cell to be \emph{free} if it is not glued to the edge of another cell. A \emph{fighting fish} is a finite set of cells constructed starting with an initial cell (the \emph{head}), then attaching one by one new cells using one of the three following operations (represented in Figure~\ref{fig:gluings}) :
\begin{itemize}
	\item \emph{Upper gluing} : Let $a$ be a cell in a fish whose right upper edge is free ; we glue the left lower edge of a new cell $b$ to the right upper edge of $a$.
	\item \emph{Lower gluing} : Let $a$ be a cell in a fish whose right lower edge is free ; we glue the left upper edge of a new cell $b$ to the right lower edge of $a$.
	\item \emph{Double gluing} : Let $a$, $b$ and $c$ three cells in a fish such that $b$ (resp. $c$) has its left lower (resp. upper) edge glued to the right upper (resp. lower) edge of $a$, and the right lower (resp. upper) edge of $b$ (resp. $c$) is free ; we glue both the left upper and left lower edges of a new cell $d$ respectively to the right lower edge of $b$ and to the right upper edge of $c$.
\end{itemize}

The \emph{size} of a fighting fish is the number of its free lower or upper edges, that is half of the length of its boundary. We denote by $\mathcal{FF}$ the class of {\em fighting fish}. Moreover, we denote by ${ff}_n= |\mathcal{FF}_n|$  the number of fighting fish with size $n$ and by  ${ff}_{i,j}=|\mathcal{FF}_{i,j}|$ the number of fighting fish with $i$ left lower free edges (or $i$ left upper free edges) and $j$ right lower free edges (or $j$ right upper free edges).

While the description of fighting fish is iterative, we are interested in these objects independently of the order in which they are constructed. There can then be multiple ways to grow a given fighting fish with these operations. We also want to emphasize that fighting fish are not planar objects in the sense that we cannot always fit them in the plane because some unit squares would represent two or more different cells. Still, we will present them in our two-dimensional pictures by taking care of showing which cells are glued together (Figure \ref{fig:ffexample}) provides an example of two possible representations of the same fighting fish).

\begin{figure}[H]
	\centering
	\includegraphics[page=17,scale=0.7]{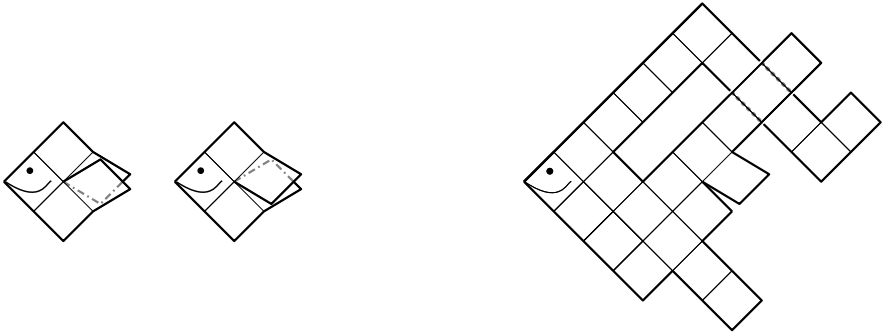}
	\caption{Two representations of the same fighting fish of size 6 and a fighting fish of size 23.}
	\label{fig:ffexample}
\end{figure}

Let us call \emph{nose} the leftmost point of the head of a fighting fish. Another way to think of a fighting fish is to perform its counterclockwise tour : if we follow the boundary of a fighting fish counterclockwise, starting from the nose, we encounter all its free edges once until getting back to it. By encoding each type of free edge by a step in $\{E, N, W, S\}$ : $E=(1,0)$ (resp. $S=(0,-1)$) for a left lower (resp. upper) one and $N=(0,1)$ (resp. $W=(-1,0)$) for a right lower (resp. upper) one obtains an excursion (walk starting and ending at the origin) on the square lattice $\mathbb{Z}^2$ confined to the quadrant $\{x,y\geq 0\}$. Then we can see fighting fish as excursions or words on the alphabet $\{E, N, W, S\}$, obtained from the word $ENWS$ using the three operations :
\begin{itemize}
	\item Upper gluing : replace a subword $W$ by $NWS$.
	\item Lower gluing : replace a subword $N$ by $ENW$.
	\item Double gluing : replace a subword $WN$ by $NW$.
\end{itemize}

For a word $w \in \{E,N,W,S,V\}^*$, we define its \emph{latitude} and \emph{longitude} to be the $x$ and $y$ coordinates of the endpoint of its corresponding walk on $\mathbb{N}^2$ starting from $(0,0)$. We then have :
\begin{align*}
\lat(w) &= |w|_N - |w|_S \\
\Long(w) &= |w|_E - |w|_W
\end{align*}

\subsection{Planar maps}

Planar maps are classical objects in combinatorics (see~\cite{chapterMaps} for more details) and while many equivalent definitions are possible, we present them as graph embeddings on the plane in order to be as graphical as possible in our setting.

\begin{definition}
	A {\em planar map} is a proper embedding of a connected graph on the plane, defined up to continuous deformations. A planar map splits the plane into \emph{edges}, \emph{vertices} and \emph{faces}, where faces are the connected components of the plane deprived of edges and vertices. A \emph{corner} is an incidence between a vertex and a face. As usual we consider {\em rooted planar maps}, where a corner on the boundary of the infinite face is distinguished. This corner is called the root, while the infinite face is also called the root face, the edge preceeding the root in counterclockwise order on the boundary of the infinite face is called the root edge  and the vertex in the root corner is called the root vertex. We denote by $\mathcal{M}$ the class of rooted planar maps (see Figure~\ref{Fig:SepAndNSMaps}$(a)$) for an example), by $\mathcal{M}_n$ the set of planar maps with $n$ edges, and by  $\mathcal{M}_{i,j}$ the set of planar maps with $i+1$ vertices and $j+1$ faces (that have hence $i+j$ edges from the Euler relation).\\
	\begin{figure}[H]
		\centering
		\includegraphics[page = 8, scale=.8]{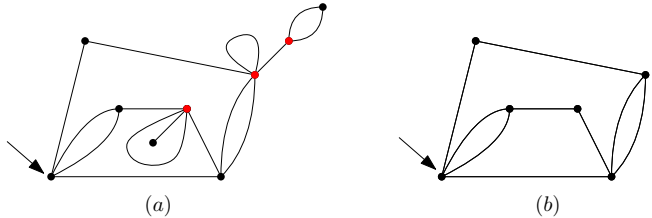}
		\caption{$(a)$ A rooted planar map with its separable vertices in red ; $(b)$ a rooted nonseparable planar map}
		\label{Fig:SepAndNSMaps}
	\end{figure}
	A {\em nonseparable planar map} is a rooted planar map without {\em separable vertices}, i.e. vertices that can be cut to disconnect the map (see Figure~\ref{Fig:SepAndNSMaps}$(b)$ for an example). We denote by $\mathcal{NS}$ the class of nonseparable planar maps, by $\mathcal{NS}_n$ the set of nonseparable planar maps with $n$ edges, and by $\mathcal{NS}_{i,j}$ the set of nonseparable planar maps with $i+1$ vertices and $j+1$ faces.\\
	For a rooted planar map $\FF{M}$, its \emph{dual map} $\FF{M}^*$ (see Figure \ref{dualmap}) is the planar map whose vertices are the faces of $\FF{M}$ whose set of edges is built in the following way : to each edge $e$ of $\FF{M}$, if we denote by $f_1$ and $f_2$ the two faces adjacent to $e$, there is an edge $e^*$ in $\FF{M}^*$ linking the vertices $v_{f_1}$ and $v_{f_2}$ respectively corresponding to faces $f_1$ and $f_2$ of $\FF{M}$. A corner corresponding to an incidence between a vertex and a face, the root corner of $\FF{M}^*$ is chosen to be the same vertex-face incidence taken as root for $\FF{M}$ (interchanging the roles of face and vertex).
	\begin{figure}[H]
		\centering
		\includegraphics[page = 19, scale=.8]{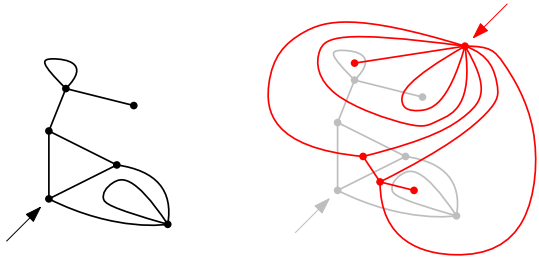}
		\caption{A rooted planar map in black then gray with its dual map drawn over it in red}
		\label{dualmap}
	\end{figure}
\end{definition}

Let us remark that because it is embedded on the plane, a planar map defines a cyclic ordering of the edges and the corners around each vertex and the same is true around each face. Hence the edge following a corner in counterclockwise direction around a vertex is well-defined as well as the corner following that edge around the same vertex. We will call them respectively the vertex-following edge and the vertex-following corner of a corner. We define similarly the face-following edge and the face-following corner of a given corner, but this time in a clockwise direction around the face.

In the following sections, we will endow rooted planar maps with a particular spanning tree called the \emph{rightmost depth-first search spanning tree} and we will denote it for short as the rightmost DFS spanning tree. For a rooted planar map $\FF{M}$, it is the following spanning tree $T$ obtained by an exploration of the corners of $\FF{M}$ using the following procedure :
\begin{itemize}
	\item Initialization : Set the tree $T$ as the tree containing the root vertex of $\FF{M}$ and no edge, and set the active corner to be the root corner. Set also the set of visited edges to be empty.
	\item Core : Consider the active corner $c$ : it is incident to a unique vertex $v$ and to a unique face $f$, and we denote by $e$ the vertex-following edge of $c$. There are 4 possible cases :
	\begin{itemize}
		\item if $e$ is not a visited edge and the face-following corner $c'$ of $c$ is incident to a vertex that does not belong to $T$, then add the edge $e$ to $T$ and to the set of visited edges and set the active corner to be $c'$ ;
		\item if $e$ is not a visited edge and the face-following corner of $c$ is incident to a vertex that belongs to $T$, then add the edge $e$ to the set of visited edges and set the active corner to be the vertex-following corner of $c$ ;
		\item if $e$ is a visited edge and $e$ is an edge of $T$, then set the active corner to be the face-following corner of $c$ ;
		\item if $e$ is a visited edge and $e$ is not an edge of $T$, then set the active corner to be the vertex-following corner of $c$.
	\end{itemize}
	Repeat until the active corner is the root corner again.
	\item End : Return $T$
\end{itemize} 

\begin{figure}[H]
	\centering
	\includegraphics[page = 22, scale=1]{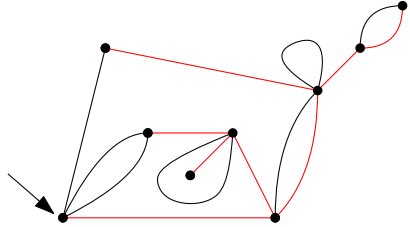}
	\caption{The rightmost depth-first search spanning tree of a rooted planar map}
	\label{RDFS-tree}
\end{figure}

Let us prove that $T$ is indeed a spanning tree of the map $\FF{M}$. First, $T$ is a tree, because every edge inserted in it during the procedure leads to a vertex that was not visited before, hence creating no cycle. Also, every corner around each vertex of $T$ is explored, and every edge is visited two times during the procedure : it is true at the leaves of $T$ and for the edges leading to it, so it is also true for vertices of $T$ having only leaves children, and so on up to the root. This property coupled with the fact that $\FF{M}$ is connected ensures that $T$ is spanning. A consequence is that all corners of the map are encountered during the exploration.

\subsection{Mullin's encoding and our result}

We now recall an encoding of tree-rooted planar maps due to Lehman \cite{Le70}, see also~\cite{walsh}, as shuffles of two Dyck words, or parenthesis-integer systems in their notation. To be precise about the origin of the encoding, Mullin was the first to notice that tree-rooted planar maps could be encoded as intertwined Catalan structures in his article \cite{Mullin} about the enumeration of tree-rooted planar maps. Our presentation of the result follows closely the approach of Schaeffer in a book chapter about planar maps \cite{chapterMaps}, and we will adopt the same denomination of Mullin's encoding.\\

Let $\FF{M}$ be a rooted planar map, that we endow with a spanning tree $T$ which we consider rooted at the same corner as $\FF{M}$. The couple $(\FF{M},T)$ is called a \emph{tree-rooted planar map}. The \emph{dual} of a tree-rooted planar map $(\FF{M},T)$ is the tree-rooted planar map $(\FF{M}^*,T^*)$, with $\FF{M}^*$ being the dual map of $M$ and $T^*$ the spanning tree of $\FF{M}^*$ formed by the edges dual of the ones in $\FF{M}$ that do not belong to $T$. We create the \emph{blossoming tree} associated to the couple $(\FF{M},T)$ by performing the \emph{counterclockwise tour} of $\FF{M}$ around $T$ : we start from the root corner of $\FF{M}$ then visit all the corners of $\FF{M}$ by traveling on the border of the tree $T$ in counterclockwise direction until we get back to the root corner. Each time we encounter an edge that does not belong to $T$ we put an opening or closing stem and we continue traveling around the stem: when we encounter such an edge for the first time, we put an opening stem on the considered vertex along the considered edge, otherwise we put a closing stem in the same manner. We then forget the edges not belonging to $T$ to get the blossoming tree associated to $(\FF{M},T)$ : it is the rooted plane tree $T$ decorated on its vertices with opening and closing stems. We define then the \emph{counterclockwise code} $\Gamma(\FF{M},T)$ of the tree-rooted planar map $(\FF{M},T)$ as the word in $\{E,N,W,S\}^*$ obtained via the following procedure : perform the counterclockwise tour of the associated blossoming tree starting from the root corner, then encode the first visit of an edge in $T$ by $E$, the second visit of an edge in $T$ by $W$, an opening stem by $N$ and a closing stem by $S$, and stop when returning at the root corner. Such a word in $\{E,N,W,S\}^*$ can be seen as a walk on the square lattice $\mathbb{Z}^2$ by sending each letter to its corresponding unit cardinal step. Let us remark that it is also to be seen as a shuffle of two Dyck words, one on the letters $\{E,W\}$, the other on the letters $\{N,S\}$. The \emph{dual} of a word in $\{E,N,W,S\}^*$ is obtained by reversing the word and replacing an occurrence of $E$ (resp. $N$, $W$, $S$) by an occurrence of $S$ (resp. $W$, $N$, $E$) : it corresponds to reversing the timeline of the corresponding walk on $\mathbb{Z}^2$ while symmetrizing it with respect to the line $y=x$.

\begin{figure}[H]
	\centering
	\includegraphics[page=11,scale=0.7]{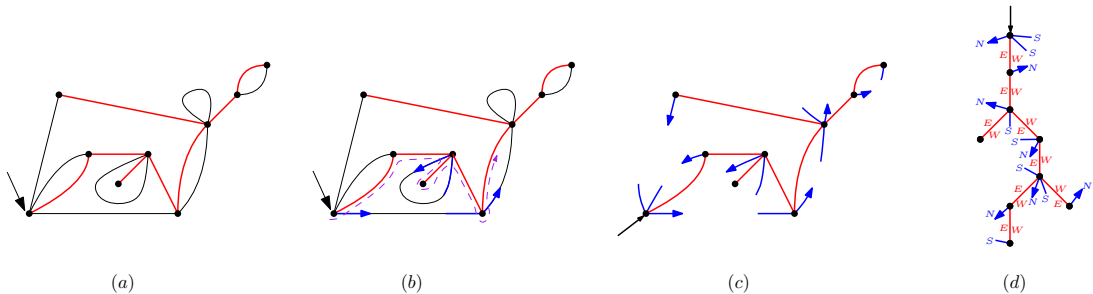}
	\caption{$(a)$ A tree-rooted planar map ; $(b)$ the beginning of the counterclockwise tour ; $(c)$ the associated blossoming tree ; $(d)$ the blossoming tree with the letters of the counterclockwise word.}
	\label{fig:mullin}
\end{figure}

In~\cite{walsh} Walsh and Lehman extend to higher genus Mullin's decomposition of tree-rooted maps and explicitely describe it in terms of shuffles of Dyck words.
 While their description of his bijection is made in a more word theoretic fashion, we state an equivalent formulation of his result here, and we refer to the review of Schaeffer in \cite{chapterMaps} (subsection 1.2.1) for a more in-depth treatment of the subject.

\begin{thm}(Mullin~\cite{Mullin})
	The counterclockwise code is a bijection between tree-rooted planar maps with $n$ edges and excursions of length $2n$, starting from the origin and confined to the first quadrant. This bijection also preserves duality, which means that the following diagram commutes :
	\[\begin{tikzcd}[arrows={|->},row sep=30pt, column sep=80pt]
	(\FF{M},T) \arrow{r}{\text{map duality}} \arrow[swap]{d}{\Gamma} & (\FF{M}^*,T^*) \arrow{d}{\Gamma} \\
	\Gamma(\FF{M},T) \arrow{r}{\text{excursion duality}} & \Gamma(\FF{M}^*,T^*)
	\end{tikzcd}
	\]
	
	Concerning statistics, a tree-rooted planar map $(M,T)$ having $i$ vertices, $j$ faces has a counterclockwise code with $i$ $E$ steps and $j$ $N$ steps.
\end{thm}

We organize the next sections as follows. In Section~\ref{FFvsNSPM} we give a bijection between fighting fish and nonseparable planar maps. The bijection is based on a new decomposition of fighting fish that is isomorphic to the classical one of nonseparable planar maps. In Section~\ref{GFFvsPM} we introduce the class of generalized fighting fish, counted by the same numbers as rooted planar maps, and we also give a decomposition isomorphic to the classical one of rooted planar maps and a direct bijection. The very enlightening result is that, for both the bijections involving nonseparable planar maps and rooted planar maps, the corresponding fighting fish through the bijections is a particular Mullin encoding of the original map, the one associated with the rightmost depth-first search spanning tree.

\section{A bijection between fighting fish and nonseparable planar maps}\label{FFvsNSPM}

 The purpose of this section is to present a bijection between fighting fish and nonseparable planar maps, already announced in~\cite{fpsac22}.

\subsection{A decomposition of fighting fish}\label{DecF}

 The  \emph{jaw} $\jaw(\FF{F})$ is, starting from the nose in counterclockwise order,  the length of the first sequence of left lower free edges.

We now present a decomposition of fighting fish that has been introduced in a companion paper exploring the connection between fighting fish and synchronized Tamari intervals~\cite{EFF}. Let us present the two operations involved in this decomposition :

\begin{itemize}
	\item the \emph{concatenation} $\FF{F}_1 \odot \FF{F}_2$ (see Figure \ref{fig:fhconc}) of two fighting fish $\FF{F}_1$ and $\FF{F}_2$ is obtained by gluing the right lower free edge ending the jaw of $\FF{F}_1$ to the leftmost upper free edge of $\FF{F}_2$.
	
	\begin{figure}[H]
		\centering
		\includegraphics[page=6,scale=0.6]{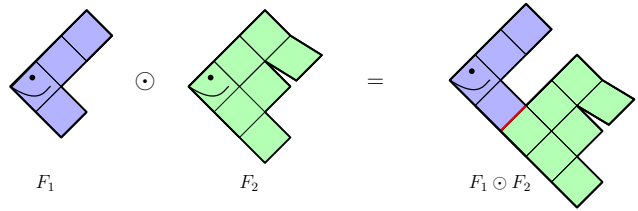}
		\caption{Concatenation of fighting fish}
		\label{fig:fhconc}
	\end{figure}
	
	\item the \emph{$i$-augmentation} $\maltese_i(\FF{F})$ of a fighting fish $\FF{F}$ (see Figure \ref{fig:fhaug}), for $i$ being an integer between 1 and $\jaw(\FF{F})$, is obtained in the following way : we glue the left lower free edge of each one of the first $i$ cells of the jaw of $\FF{F}$ to a new cell and then glue the right lower edge and the left upper edge of all pairs of adjacent new cells.
\end{itemize}

\begin{figure}[H]
	\centering
	\includegraphics[page=7,scale=0.75]{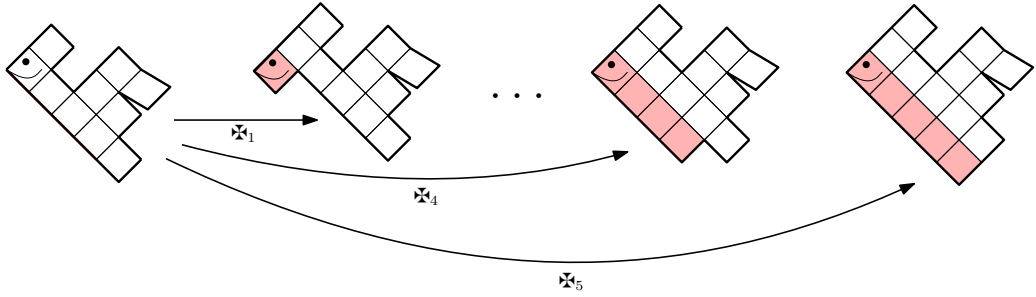}
	\caption{Some augmentations of a fighting fish}
	\label{fig:fhaug}
\end{figure}

\begin{prop}
	For every fighting fish $\FF{F}$, $\FF{F}_1$ and $\FF{F}_2$, and for every $1 \leq i \leq \jaw(\FF{F})$, we have :
	\begin{align*}
		\size(\maltese_i(\FF{F})) &= \size(\FF{F})+1\\
		\jaw(\maltese_i(\FF{F})) &= i\\
		\size(\FF{F}_1 \odot \FF{F}_2) &= \size(\FF{F}_1)+\size(\FF{F}_2)-1\\
		\jaw(\FF{F}_1 \odot \FF{F}_2) &= \jaw(\FF{F}_1) + \jaw(\FF{F}_2)
	\end{align*}
\end{prop}

\begin{figure}[H]
	\centering
	\includegraphics[scale=0.75]{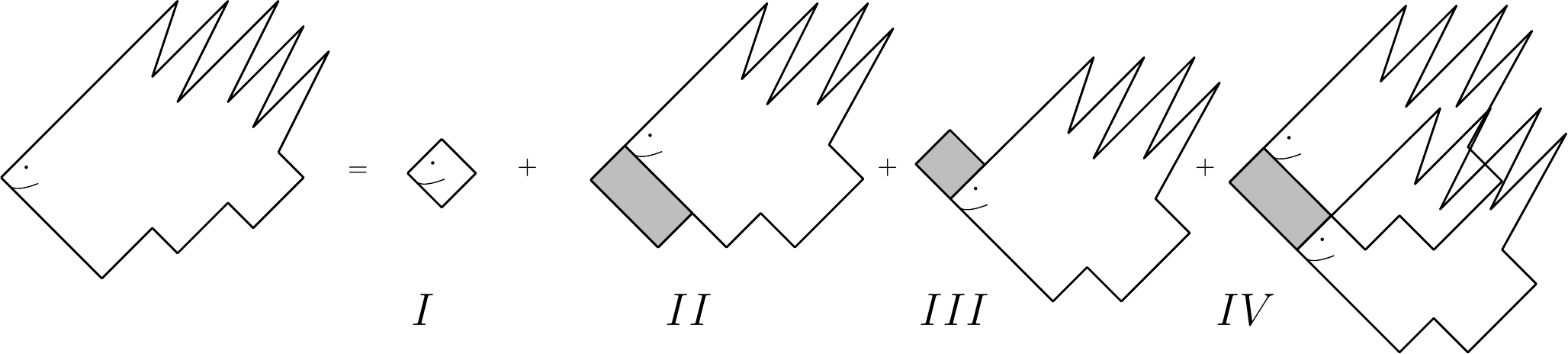}
	\caption{A decomposition on fighting fish}
	\label{fig:DecFish}
\end{figure}

Consider now a fighting fish $\FF{F}$. Then exactly one of the following cases occur (see Figure~\ref{fig:DecFish}) :
\begin{itemize}
	\item Case I : $\FF{F} = ENWS$ is the Head (only one empty component).
	\item Case II : $\FF{F} = \maltese_i(\FF{F}_1)$ for some fighting fish $\FF{F}_1$ and some $1 \leq i \leq \jaw(\FF{F}_1)$ (only one non-empty component).
	\item Case III : $\FF{F} = ENWS \odot \FF{F}_2$ for some fighting fish $\FF{F}_2$ (more than two components, the last one being empty).
	\item Case IV : $\FF{F} = \maltese_i(\FF{F}_1) \odot \FF{F}_2$ for some fighting fish $\FF{F}_1$ and $\FF{F}_2$ and some $1 \leq i \leq \jaw(\FF{F}_1)$ (more than two components, the last one being non-empty).
\end{itemize}

\subsection{A decomposition of nonseparable planar maps}

The \emph{reduced outer degree} $\out(\FF{M})$ of a nonseparable planar map $\FF{M}$ is the number of non-root corners incident to the root face. Note that we can assign an integer from $1$ to $\out(\FF{M})$ to each non-root corner in the root face according to its position in the counterclockwise tour of the root face. We define now two operations on nonseparable planar maps :

\begin{itemize}
	\item the \emph{concatenation} $\FF{M}_1 \odot \FF{M}_2$ (see Figure \ref{fig:mapconc}) of two nonseparable planar maps $\FF{M}_1$ and $\FF{M}_2$ having respective root edges $e_1 = u_1 \rightarrow v_1$ and $e_2 = u_2 \rightarrow v_2$ is obtained by identifying $v_2$ and $u_1$, then deleting the root edge of $\FF{M}_1$ and $\FF{M}_2$ and creating a new root edge from $u_2$ to $v_1$.
	
	\begin{figure}[H]
		\centering
		\includegraphics[page=9,scale=0.8]{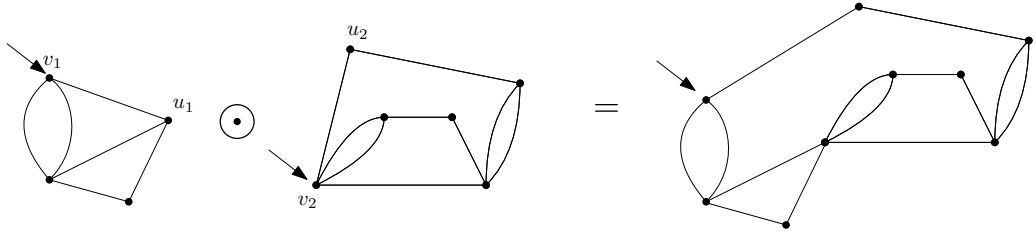}
		\caption{Concatenation of nonseparable planar maps.}
		\label{fig:mapconc}
	\end{figure}
	
	\item the \emph{$i$-augmentation} $\maltese_i(\FF{M})$ of a nonseparable planar map $\FF{M}$ (see Figure \ref{fig:mapaug}), for $i$ being an integer between 1 and $\out(\FF{M})$, is obtained by adding an edge between the $i^{\rm{th}}$ non-root vertex incident to the root face and the root vertex so that the root edge of $\FF{M}$ is no more incident to the root face, and the added edge is the new root edge.
\end{itemize}

\begin{figure}[H]
	\centering
	\includegraphics[page=10,scale=0.75]{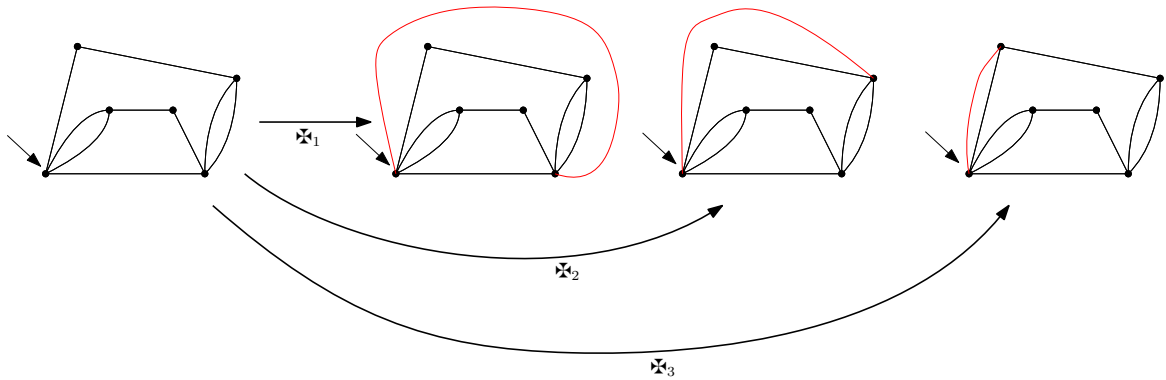}
	\caption{The three possible augmentations of some nonseparable planar map.}
	\label{fig:mapaug}
\end{figure}

\begin{prop}
	For every nonseparable planar maps $\FF{M}$, $\FF{M}_1$ and $\FF{M}_2$, and for every $1 \leq i \leq \out(\FF{M})$, we have :
	\begin{align*}
	\size(\maltese_i(\FF{M})) &= \size(\FF{M})+1\\
	\out(\maltese_i(\FF{M})) &= i\\
	\size(\FF{M}_1 \odot \FF{M}_2) &= \size(\FF{M}_1)+\size(\FF{M}_2)-1\\
	\out(\FF{M}_1 \odot \FF{M}_2) &= \out(\FF{M}_1) + \out(\FF{M}_2)
	\end{align*}
\end{prop}

\begin{figure}[H]
	\centering
	\includegraphics[scale=0.75]{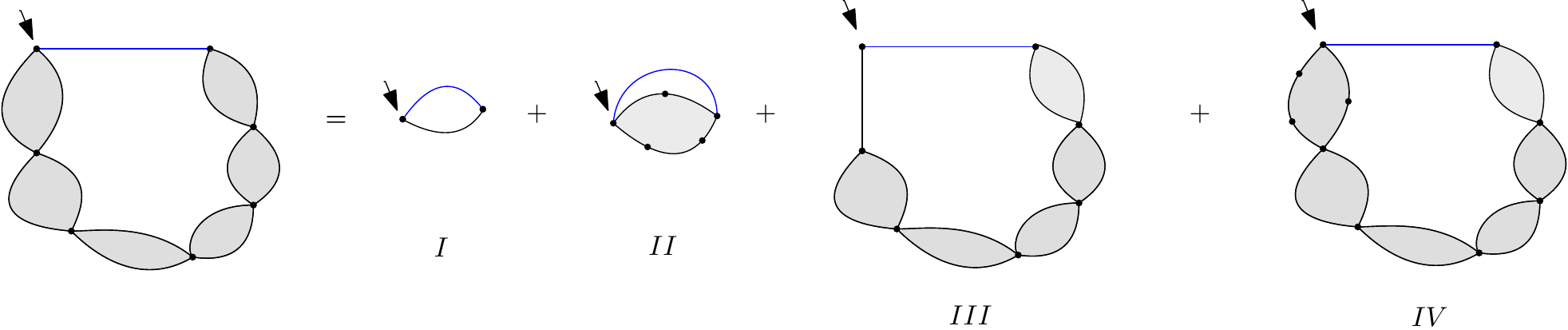}
	\caption{A decomposition on nonseparable planar maps}
	\label{fig:DecCarteNonSep}
\end{figure}

We now recall a classical decomposition of nonseparable planar maps called the \emph{series decomposition} (see Figure~\ref{fig:DecCarteNonSep}). It was first stated by Brown in \cite{Brown}. Consider a nonseparable planar map $\FF{M}$ and delete its root edge to obtain an unrooted planar map $\FF{M}'$ that may have separable vertices. By cutting $\FF{M}'$ at its separable vertices, we obtain a tuple of planar maps that are either nonseparable or the one edge planar map with two distinct vertices (that we call  \emph{Bridge} and note $B$). Then exactly one of the following four cases occurs :
\begin{itemize}
	\item Case I : $\FF{M}$ is the only nonseparable map with two edges (that we will denote by $D$ in the following) ($\FF{M}'$ has only one component and it is a Bridge).
	\item Case II : $\FF{M} = \maltese_i(\FF{M}_1)$ for some nonseparable planar map $\FF{M}_1$ and some $1 \leq i \leq \out(\FF{M}_1)$ ($\FF{M}'$ has only one component and it is not a Bridge).
	\item Case III : $\FF{M} = B \odot \FF{M}_2$ for some nonseparable planar map $\FF{M}_2$ ($\FF{M}'$ has more than two components and the last one is a Bridge).
	\item Case IV : $\FF{M} = \maltese_i(\FF{M}_1) \odot \FF{M}_2$ for some nonseparable planar maps $\FF{M}_1$ and $\FF{M}_2$ and some $1 \leq i \leq \out(\FF{M}_1)$ ($\FF{M}'$ has more than two components and the last one is not a Bridge).
\end{itemize}

\subsection{The bijection}

In the last two subsections, we presented two recursive decompositions for fighting fish and for nonseparable planar maps. Since these two decompositions are isomorphic, we are now able to define a bijection $\Phi$ from nonseparable planar maps to fighting fish. Let $\FF{M}$ be a nonseparable planar map, we examine all possible cases :
\begin{itemize}
	\item Case I : if $\FF{M} = D$, we set $\Phi(\FF{M}) = ENWS$.
	\item Case II : if $\FF{M} = \maltese_i(\FF{M}_1)$ for some nonseparable planar map $\FF{M}_1$ and some $1 \leq i \leq \out(\FF{M}_1)$, we set $\Phi(\FF{M}) = \maltese_i(\Phi(\FF{M}_1))$.
	\item Case III : if $\FF{M} = D \odot \FF{M}_2$ for some nonseparable planar map $\FF{M}_2$, we set $\Phi(\FF{M}) = ENWS \odot \Phi(\FF{M}_2)$.
	\item Case IV : if $\FF{M} = \maltese_i(\FF{M}_1) \odot \FF{M}_2$ for some nonseparable planar maps $\FF{M}_1$ and $\FF{M}_2$ and some $1 \leq i \leq \out(\FF{M}_1)$, we set $\Phi(\FF{M}) = \maltese_i(\Phi(\FF{M}_1)) \odot \Phi(\FF{M}_2)$.
\end{itemize}

\begin{thm}
	$\Phi$ is a bijection between nonseparable planar maps with $n$ edges, $i+1$ vertices, $j+1$ faces and outer degree $k$ and fighting fish of size $n$ with $i$ left lower free edges, $j$ right lower free edges and jaw length $k$. Moreover, for a nonseparable planar map $\FF{M}$, $\Phi(\FF{M})$ is the Mullin encoding of the tree-rooted planar map $(\FF{M},T)$ where $T$ is the rightmost depth-first search spanning tree of $\FF{M}$. As a consequence, $\Phi$ preserves duality.
\end{thm}

\begin{figure}[H]
	\centering
	\includegraphics[page = 20, scale=0.7]{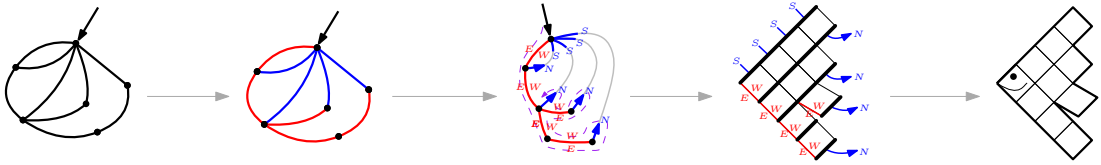}
	\caption{The bijection $\Phi$ applied to some nonseparable planar map}
	\label{fig:bijnonsep}
\end{figure}

\begin{proof}
	Since the two recursive decompositions are isomorphic, $\Phi$ is a well-defined bijection between nonseparable planar maps with $n$ edges and out degree $k$ and fighting fish of size $n$ and jaw length $k$. We proceed by induction on the size of nonseparable planar maps to prove that $\Phi$ maps nonseparable planar maps to their rightmost depth-first search Mullin encoding. Let $\FF{M}$ be a nonseparable planar map. We distinguish between the four possible cases :
	\begin{itemize}
		\item Case I : if $\FF{M} = D$, then $\Phi(\FF{M}) = ENWS$ is indeed the Mullin encoding of $(\FF{M},T)$ where $T$ is the rightmost DFS tree of $\FF{M}$.
		\item Case II : if $\FF{M} = \maltese_i(\FF{M}_1)$ for some nonseparable planar map $\FF{M}_1$ and some $1 \leq i \leq \out(\FF{M}_1)$, then by the induction hypothesis, $\Phi(\FF{M}_1)$ is the Mullin encoding of $\FF{M}_1$ endowed with its rightmost DFS spanning tree $T_1$. Then the rightmost DFS spanning tree of $\FF{M}$ is also $T_1$ since $\FF{M}$ is obtained by adding an edge towards the root, that is an already visited vertex. Hence, if we write $\Phi(\FF{M}_1) = E^{k_1} \FF{G}$ with $k_1 = \out(\FF{M}_1)$, the Mullin encoding of $(\FF{M},T_1)$ is  $E^i N E^{k_1-i} \FF{G} S$, that is exactly $\maltese_i(\Phi(\FF{M}_1)) = \Phi(\FF{M})$.
		\item Case III : if $\FF{M} = D \odot \FF{M}_2$ for some nonseparable planar map $\FF{M}_2$, then by the induction hypothesis, $\Phi(\FF{M}_2)$ is the Mullin encoding of $\FF{M}_2$ endowed with its rightmost DFS spanning tree $T_2$. The rightmost DFS spanning tree $T$ of $\FF{M}$ consists of a root vertx with a pending edge on which $T_2$ is attached. Writing $\Phi(\FF{M}_2) = \FF{G} S$, the Mullin encoding of $(\FF{M},T)$ is then $E\FF{G} W S$, that is exactly $ENWS \odot \Phi(\FF{M}_2) = \Phi(\FF{M})$.
		\item Case IV : if $\FF{M} = \maltese_i(\FF{M}_1) \odot \FF{M}_2$ for some nonseparable planar maps $\FF{M}_1$ and $\FF{M}_2$ and some $1 \leq i \leq \out(\FF{M}_1)$, then by the induction hypothesis, $\Phi(\FF{M}_1)$ and $\Phi(\FF{M}_2)$ are the respective Mullin encodings of $\FF{M}_1$ and $\FF{M}_2$ endowed with their respective rightmost DFS spanning tree $T_1$ and $T_2$. Then the rightmost DFS spanning tree $T$ of $\FF{M}$ is obtained by gluing the root of $T_2$ onto the $k$-th corner of the righmost branch of $T_1$, with $k+ \out(T2)= \out(T_1)$. Let us write $\Phi(\FF{M}_1) = E^{k} \FF{G}_1$ and $\Phi(\FF{M}_2) = \FF{G}_2 S$. Performing the counterclockwise tour of $\FF{M}$ around $T$, we get the Mullin encoding of $(\FF{M},T)$ $E^{k} \FF{G}_2 \FF{G}_1 S$, which is exactly $\maltese_i(\Phi(\FF{M}_1)) \odot \Phi(\FF{M}_2) = \Phi(\FF{M})$.
	\end{itemize}
\end{proof}

\subsection{Enumeration results}

For the sake of completeness, we recall now some enumeration results, which have been known for a long time for nonseparable planar maps (see \cite{browntutte} for example). Concerning fighting fish, the counting has been made analytically in \cite{ff2}, and the preceeding bijection makes it more enlightening combinatorially.

\begin{thm}
	The number of rooted nonseparable planar maps having $n +1$ edges, or equivalently the number of fighting fish of size $n+1$ is :
	\[\frac{2(3n)!}{(n+1)!(2n+1)!}\]
	The number of rooted nonseparable planar maps having $i+1$ vertices and $j+1$ faces, or equivalently the number of fighting fish with $i+1$ $E$ steps and $j+1$ $N$ steps is :
	\[\frac{(2i+j-2)!(2j+i-2)!}{i!j!(2i-1)!(2j-1)!}\]
\end{thm}

\section{Generalized fighting fish}\label{GFFvsPM}

Let us now present an alternative definition of fighting fish in terms that will enable us to generalize it. 
\begin{prop}
Fighting fish are exactly finite sets of cells that can be built starting from the head, then attaching new cells using one of the following two operations :
\begin{itemize}
	\item \emph{Upper strip gluing} : Let $a_1, a_2, ..., a_k$, $k \geq 1$ be cells such that $a_i$ has its right lower edge glued to the left upper edge of $a_{i+1}$ for all $1 \leq i \leq k-1$, and such that all $a_i$ have their right upper edge free ; for every $i$, we glue the right upper edge of $a_i$ to the left lower edge of a new cell $b_i$ while gluing the right lower edge of $b_i$ to the left upper edge of $b_{i+1}$ for $1 \leq i \leq k-1$.
	\item \emph{Lower strip gluing} : Let $a_1, a_2, ..., a_k$, $k \geq 1$ be cells such that $a_i$ has its right upper edge glued to the left lower edge of $a_{i+1}$ for all $1 \leq i \leq k-1$, and such that all $a_i$ have their right lower edge free ; for every $i$, we glue the right lower edge of $a_i$ to the left upper edge of a new cell $b_i$ while gluing the right upper edge of $b_i$ to the left lower edge of $b_{i+1}$ for $1 \leq i \leq k-1$.
\end{itemize}
\end{prop}

\begin{proof}
	First, the upper strip gluing of $k$ cells $b_1, ..., b_k$ to $a_1, ..., a_k$ can be obtained by performing an upper gluing of $b_1$ on $a_1$, then performing $k-1$ double gluings of $b_{i+1}$ on $a_{i+1}$ and $b_i$, for $1 \leq i \leq k-1$. Similarly, a lower strip gluing of $k$ cells can be obtained with a lower gluing and $k-1$ double gluings. Hence every finite set of cells that can be obtained starting from the head and performing upper and lower strip gluings is a fighting fish.\\
	We will now prove that every fighting fish can be obtained starting from the head and performing upper and lower strip gluings by induction on the number $m$ of cells in the fighting fish. For $m=1$, the only fighting fish is the head, so the property is verified. Suppose that the property is true for fighting fish having strictly less than $m \geq 2$ cells and let $\mathrm{F}$ be a fighting fish having $m$ cells. If $\mathrm{F}$ is obtained by applying an upper gluing on a fighting fish $\mathrm{F}_1$, then $\mathrm{F}_1$ has $m-1$ cells, so it is constructible using upper and lower strip gluings, and so is $\mathrm{F}$ by performing one additional upper strip gluing of $1$ cell. The case where $\mathrm{F}$ is obtained with a lower gluing on a fighting fish $\mathrm{F}_2$ is similar. If $\mathrm{F}$ is obtained with a double gluing of a cell $b$ on two cells $a$ (on the left upper edge of $b$) and $a'$ (on the left lower edge of $b$) of a fighting fish $\mathrm{F}_3$, then $\mathrm{F}_3$ has $m-1$ cells so it is constructible from the head using upper and lower strip gluings. In this construction, either $a$ is added using an upper strip gluing or $a'$ is added using a lower strip gluing, because in $\mathrm{F}_3$ (and in $\mathrm{F}$), the left lower edge of $a$ and the left upper edge of $a'$ are glued respectively to the right lower edge and to the right upper edge of the same cell $c$. If $a$ is added using an upper strip gluing in $\mathrm{F}_3$, then $\mathrm{F}$ can be obtained using the same strip construction as for $\mathrm{F}_3$, except we perform an upper strip gluing with the extra cell $b$ when we glue the upper strip containing $a$. The symmetrical statement is true if $a'$ is added using a lower strip gluing. Finally, $\mathrm{F}$ is constructible using upper and lower strip gluings and the induction follows.
\end{proof}

If we state this proposition in the word setting, we obtain the following : fighting fish are exactly the finite words on the alphabet $\{E, N, W, S\}$ that can be obtained from the word $ENWS$ using a sequence of the two operations :
\begin{itemize}
	\item Upper strip gluing : replace a subword $W^k$, $k \geq 1$, by $NW^kS$.
	\item Lower strip gluing : replace a subword $N^k$, $k \geq 1$, by $EN^kW$.
\end{itemize}

We can then build a class of walks that includes fighting fish if we allow $k$ to be $0$ in the operations of strip gluing. We will prefer from now on to work in the words or walks setting since strip gluings with $k=0$ do not add cells.

\begin{definition}
The class of \emph{generalized fighting fish}, denoted $\mathcal{GFF}$ is the set of words or walks inductively defined from the empty word $\varepsilon$ with the help of the following two types of operations :
\begin{itemize}
\item \textbf{operation $\bigtriangledown_k$, $k\geq 0$}: replace a subword $N^k$ by $EN^kW$.
\item\textbf{operation $\bigtriangleup_k$, $k\geq 0$}: replace a subword $W^k$ by $NW^kS$.
\end{itemize}
\end{definition}

\begin{figure}[H]
	\centering
	\includegraphics[page=15,width=0.5\textwidth]{figures-these}
	\caption{Operations $\bigtriangledown_k$ and $\bigtriangleup_k$}
	\label{fig:construction}\label{fig:FightingFish}
\end{figure}

As for fighting fish, the \emph{size} of a generalized fighting fish $\mathrm{F}$ is half of its length, and we have :\\
$$\size(\mathrm{F}) = \frac{1}{2} (|h|_E+|h|_N+|h|_W+|h|_S) = |h|_E+|h|_N = |h|_W+|h|_S$$

Remark that the class of fighting fish is given  by those generalized fighting fish that can be obtained from $EW$ and $NS$ using only operations $\bigtriangledown_k$ and $\bigtriangleup_k$, with $k\geq 1$. (See Figure~\ref{fig:FightingFish} for an example)

We denote by ${gf}_n=|\mathcal{GF}_{n}|$  the number of generalized fighting fish with size $n$ and by ${gf}_{i,j}=|\mathcal{GF}_{i,j}|$ the number of generalized fighting fish with $i$ $E$ steps (or $i$ $W$ steps) and $j$ $N$ steps (or $j$ $S$ steps).

Generalized fighting fish are defined inductively using the operator
$\bigtriangleup_k$ and $\bigtriangledown_k$ on words but it is useful
to take advantage of their geometric properties as lattice paths. In
particular the following lemma will be useful, where $P^<$ denote the
set of points that can be reached from a point $P$ by $N$ and $E$
steps (ie the $45$ tilted upper-right quarter plane with origin at
$P$):

\begin{lem}\label{Lem:BorderProp}
  Let $\FF{F}$ be a generalized fighting fish with origin $O$. Then $\FF{F}$ is
  included in $O^<$. Moreover if $\FF{F}=\FF{F}_1\odot \FF{F}_2$ with $P$ the
  endpoint of $\FF{F}_1$ then $\FF{F}_2$ remains in $P^<$ or exits from the
  upper side of $P^<$ (that is with a $W$ step, or with a $S$ step originating from $P$). 
\end{lem}

\begin{proof}
  The proof is by recurrence: the fish $\FF{F}$ can be obtained from a
  smaller fish $\FF{F}'$ by applying the operation $\bigtriangledown_k$ or
  $\bigtriangleup_k$ to a seed $N^k$ or $W^k$ in $\FF{F}'$. Assume that
  $\FF{F}'$ satisfies the lemma: then it stays in $O^<$ and this property
  is preserved by $\bigtriangleup_k$ or $\bigtriangledown_k$; moreover
  if $\FF{F}=\FF{F}_1\odot \FF{F}_2$ then $\FF{F}'=\FF{F}'_1\odot \FF{F}'_2$ with the seed either in
  $\FF{F}'_1$, in $\FF{F}'_2$ or spread between the two: in each case the
  second property is preserved as well.
  \end{proof}

\section{A bijection between generalized fighting fish and rooted planar maps}\label{sec1}
 
\subsection{A decomposition for generalized fighting fish}
In order to present a decomposition for generalized fighting fish we need to introduce the following operations:
\begin{itemize}
\item the concatenation $\oplus$ of two generalized fighting fish $\FF{F}_1$ and $\FF{F}_2$, that is given by $\FF{F}_1 \oplus \FF{F}_2=\FF{F}_1 E \FF{F}_2 W$ (see Figure~\ref{fig:OpGenFigFishConc}).
\begin{figure}[H]
	\centering
	\includegraphics[scale=0.7]{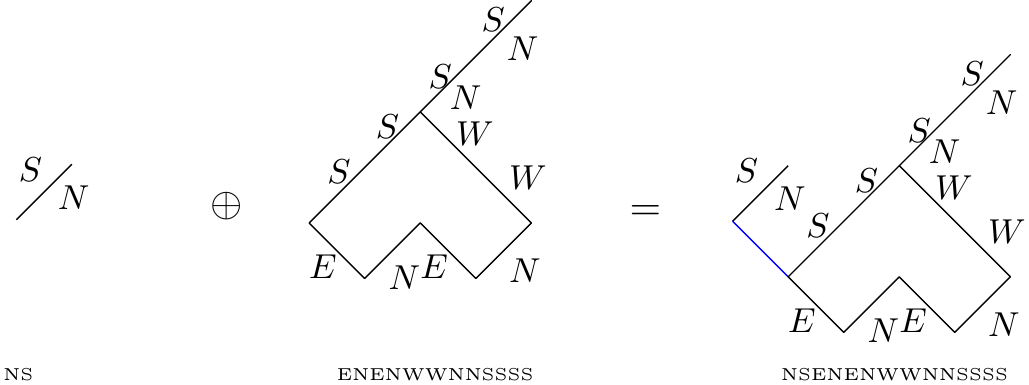}
	\caption{The concatenation of two generalized fighting fish $\FF{F}_1$ and $\FF{F}_2$.}
	\label{fig:OpGenFigFishConc}
\end{figure}
\item let us take a generalized fighting fish $\FF{F}$ and  denote by $\ell$ its number of visits at latitude $0$ minus one, then  the \emph{$i$-augmentation} $\maltese_i(\FF{F})$ of  $\FF{F}$ (see Figure \ref{fig:OpGenFigFish}), for $i$ being an integer between $0$ and $\ell$, is obtained by adding a $N$ step after step $i$ at latitude $0$ and a $S$ step at the end of the walk.
More precisely, let $\FF{G}_i$ the subwalk of $\FF{F}$ starting from $0$ and ending at the $i$-th visit, and let us denote by $\overline{\FF{G}}_i$ the remaining subwalk of  $\FF{F}$. Then $\maltese_i(\FF{F})=\FF{G}_i N \overline{\FF{G}}_i S$.
\begin{figure}[H]
	\centering
	\includegraphics[scale=0.65]{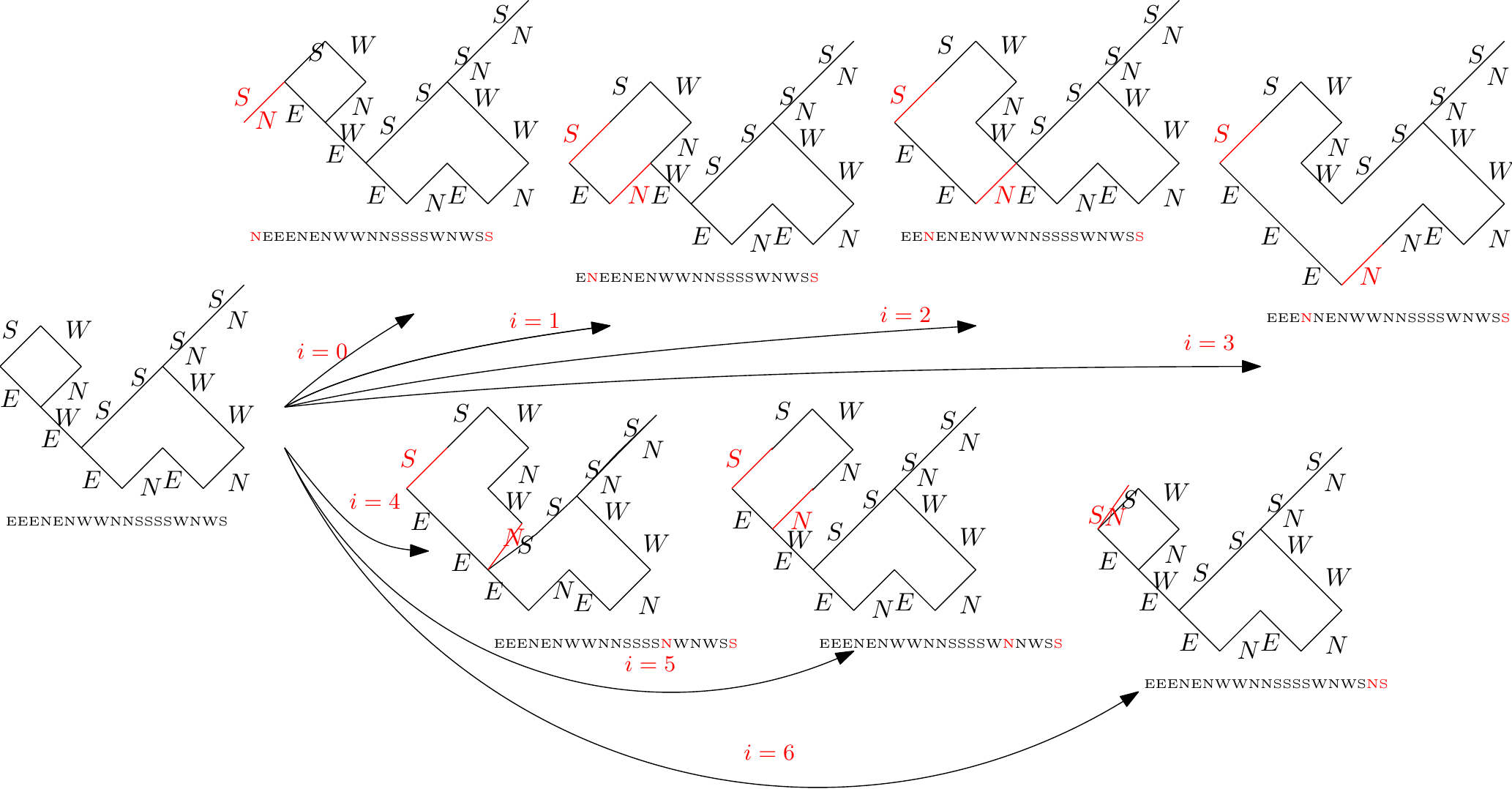}
	\caption{The $i$-augmentation of a generalized fighting fish, represented with walks and words.}
	\label{fig:OpGenFigFish}
\end{figure}
\end{itemize}

\begin{prop}
	For every generalized fighting fish $\FF{F}$ and for every $0 \leq i \leq \ell(\FF{F})$, we have :
	\begin{align*}
		\size(\maltese_i(\FF{F})) &= \size(\FF{F})+1\\
		\ell(\maltese_i(\FF{F})) &= i+1\\
	         \size(\FF{F}_1 \oplus \FF{F}_2) &= \size(\FF{F}_1)+\size(\FF{F}_2)+1\\ 		\ell(\FF{F}_1 \oplus \FF{F}_2) &= \ell(\FF{F}_1) + \ell(\FF{F}_2)+2
	\end{align*}
\end{prop}

Remark that the $i$-augmentation  of a generalized fighting fish is a generalization of $\maltese_i$ defined for fighting fish in \ref{DecF}. Indeed for a fighting fish, the number of points at latitude $0$ corresponds exactly to the jaw length minus 1.

Given a generalized fighting fish $\FF{F}$ of size $n$, then we have two possible cases: 
\begin{figure}[H]
	\centering
	\includegraphics[scale=0.4]{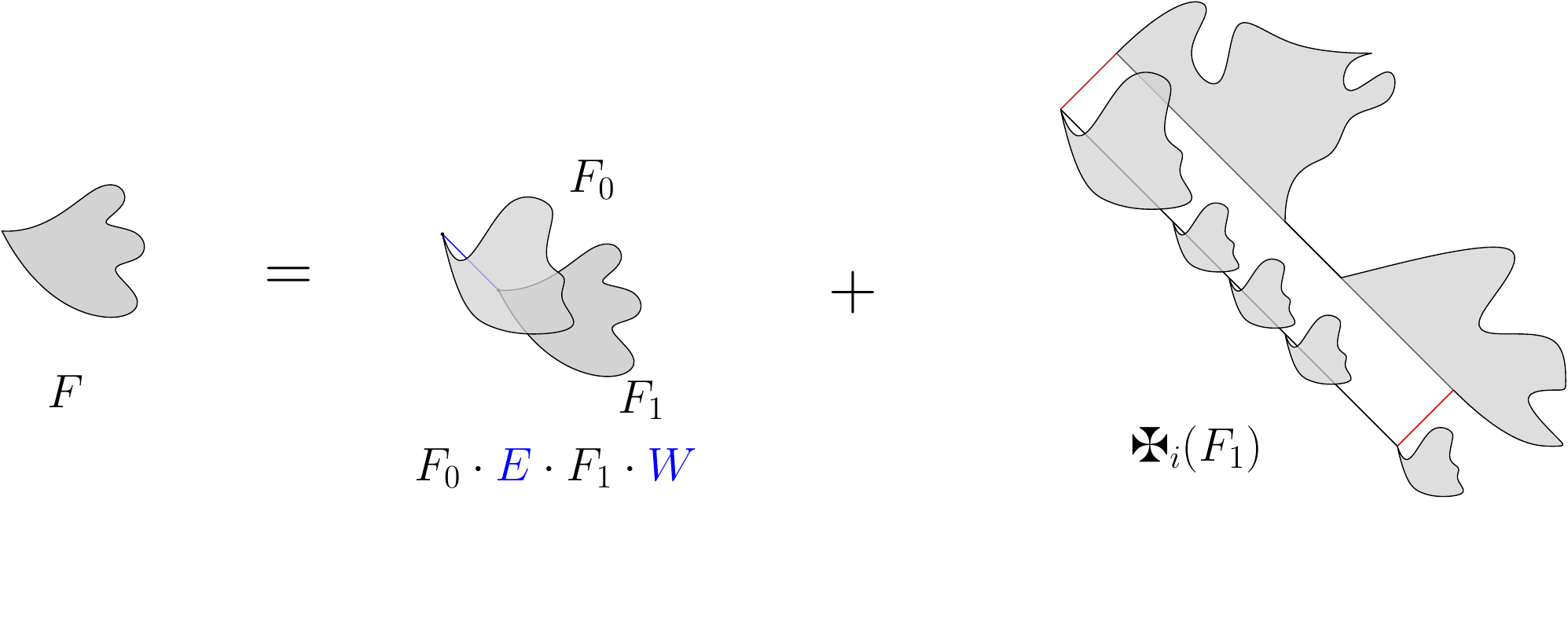}
	\caption{A decomposition for generalized fighting fish.}
	\label{Fig:GenFishDec}
\end{figure}

\begin{itemize}
\item Case I : $\FF{F}$ ends with a $W$ step. Then we consider the last step $s$ leaving the origin in the path associated with $\FF{F}$, then $\FF{F}=\FF{F}_1 s \FF{F}_2 W$.  In view of the growth operations $\bigtriangledown_k$ and $\bigtriangleup_k$ $s$ can only be a $E$ or $N$ step. If it was a $N$ step  the quadrant property enounced in lemma~\ref{fig:construction} applied to $N\cdot \FF{F}_2$
would imply that $\FF{F}_2$ would return to the origin before visiting any point with $0$ latitude. Therefore $s$ is a $E$ step and $\FF{F}=\FF{F}_1 \oplus \FF{F}_2=\FF{F}_1 E \FF{F}_2 W$ where $\FF{F}_1$ and $\FF{F}_2$ are possibly empty generalized fighting fish.

 \item Case II : $\FF{F}$ ends with a $S$ step. Then $\FF{F} = \maltese_i(\FF{F}_1)$ for some fighting fish $\FF{F}_1$ and some $0 \leq i \leq \ell(\FF{F}_1)$, where $\ell(\FF{F}_1)$ is the number of steps starting at latitude $0$ in $\FF{F}$. 
\end{itemize}

\subsection{Rooted planar maps and subclasses}

\subsection{A classical decomposition of rooted planar maps}

In order to present Tutte's classical \emph{root edge deletion} decomposition of a rooted planar map $\FF{M}$ we need to introduce the following operations:
\begin{itemize}
\item the concatenation $\FF{M}_1 \oplus \FF{M}_2$ of two planar maps $\FF{M}_1$ and $\FF{M}_2$ is obtained by adding a edge from the root corner of $\FF{M}_1$ to the root of $\FF{M}_2$ and by keeping the root of $\FF{M}_1$ (see Figure~\ref{fig:OpPlanarMapConc}).
\begin{figure}[H]
	\centering
	\includegraphics[scale=0.6]{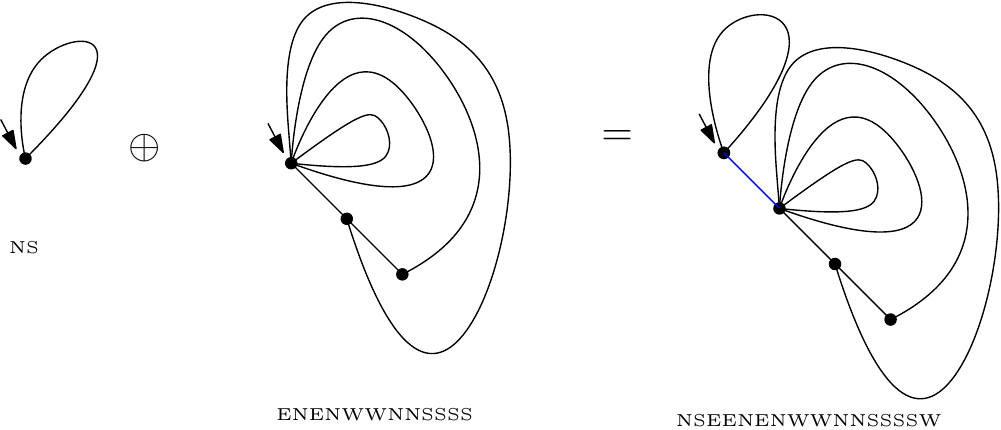}
	\caption{ The concatenation $\oplus$ of two planar maps $\FF{M}_1$ and $\FF{M}_2$ and the Mullin encodings with the rightmost spannig tree..}
	\label{fig:OpPlanarMapConc}
\end{figure}
  
\item let us denote by $c$ the number of corners in the infinite face of a rooted planar map $\FF{M}$, then  the \emph{$i$-augmentation} $\maltese_i(\FF{M})$ of  $\FF{M}$ (see Figure~\ref{fig:OpPlanarMap}), for $i$ an integer between $0$ and $c$, is obtained by adding a edge from the $i$-th corner to the right of the root of $\FF{M}$ (observe that we consider two corners around the root, the one on the left and the one on the right of the arrow indicating the root corner).
\begin{figure}[H]
	\centering
	\includegraphics[scale=0.6]{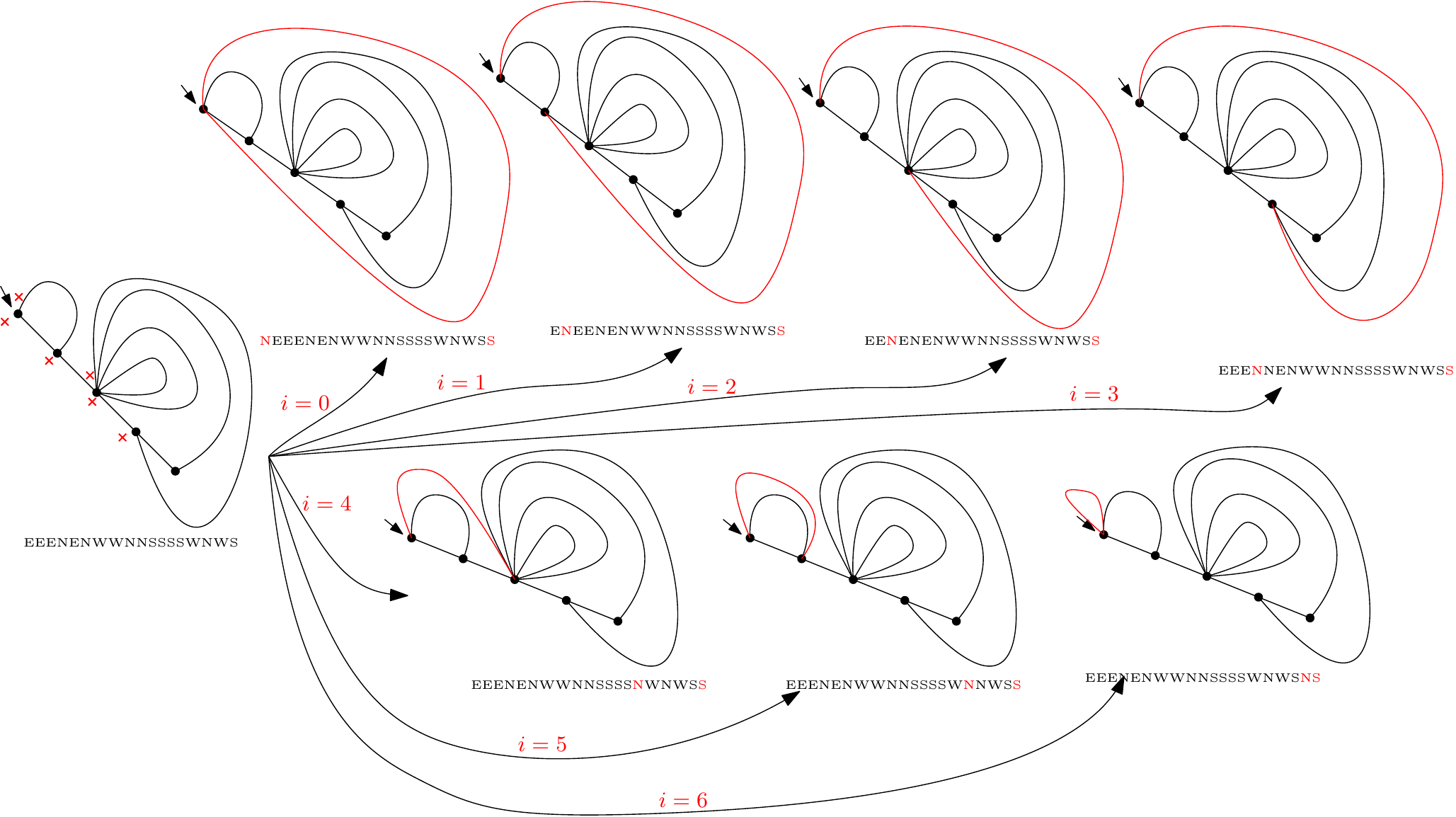}
	\caption{ The $i$-augmentation of a rooted planar map and the Mullin encodings with the rightmost spanning tree.}
	\label{fig:OpPlanarMap}
\end{figure}
\end{itemize}

\begin{figure}[ht]
	\centering
	\includegraphics[scale=0.5]{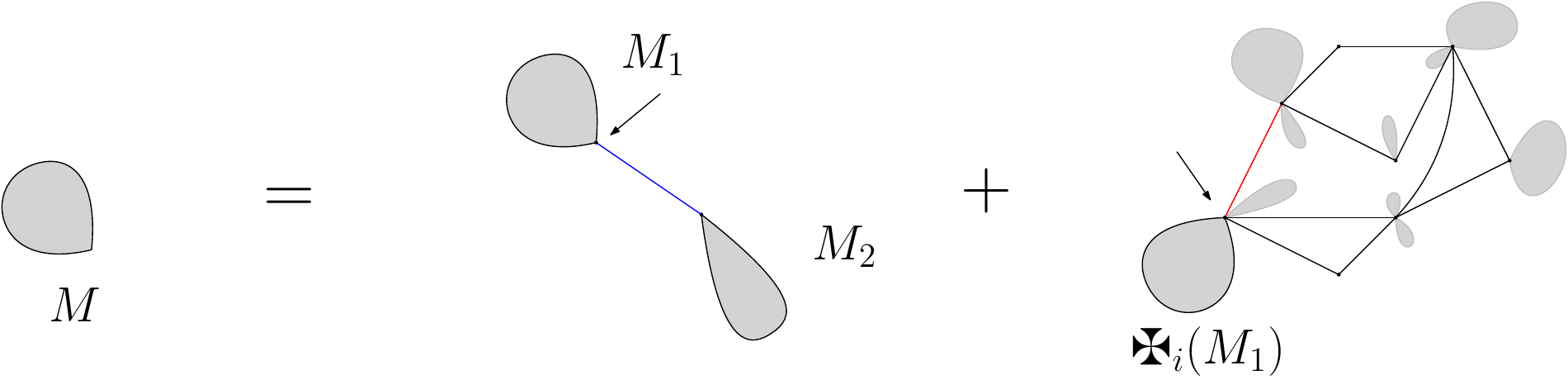}
	\caption{ A classical decomposition on rooted planar maps.}
	\label{Fig:Dec-Gen-Maps-2}
\end{figure}
Tutte's classical \emph{root edge deletion} decomposition of a rooted planar map $\FF{M}$ consists in removing the edge $e$ preceeding the root corner in counterclockwise order. There are two possible cases, represented in Figure~\ref{Fig:Dec-Gen-Maps-2}:
\begin{itemize}
\item  Case I : The edge $e$ is a disconnecting edge, so that there are two possibly empty rooted planar maps $(\FF{M}_1, \FF{M}_2)$, such that $\FF{M}=\FF{M}_1\oplus \FF{M}_2$. 
  \item  Case II : The edge $e$ is not disconnecting, then  $\FF{M} = \maltese_i(\FF{M}_1)$ for some rooted planar map $\FF{M}_1$ and some $0 \leq i \leq c(\FF{M}_1)$, where $c(\FF{M}_1)$ is the number of corners in the infinite face of $\FF{M}_1$. 

     \end{itemize}

\begin{prop}
	For every rooted  planar maps $\FF{M}$, $\FF{M}_1$ and $\FF{M}_2$, and for every $1 \leq i \leq c(\FF{M})$, we have :
	\begin{align*}
	\size(\maltese_i(\FF{M})) &= \size(\FF{M})+1\\
	c(\maltese_i(\FF{M})) &= i+1\\
	\size(\FF{M}_1 \oplus \FF{M}_2) &= \size(\FF{M}_1)+\size(\FF{M}_2)+1\\
	c(\FF{M}_1 \oplus \FF{M}_2) &= c(\FF{M}_1) + c(\FF{M}_2)+2
	\end{align*}
\end{prop}

\subsection{The bijection}

In the last two subsections, we presented a recursive decomposition for generalized fighting fish and recalled the classical decomposition for rooted  planar maps. Since these two decompositions are isomorphic, we are now able to define a bijection $\xi$ from rooted planar maps to generalized fighting fish. Let $\FF{M}$ be a rooted planar map, we examine all possible cases :
\begin{itemize}
    \item Case I : the  root edge $e$ of $\FF{M}$' is a disconnecting one, that is, if removed it disconnects $\FF{M}$ into two rooted planar maps $\FF{M}_1$ and $\FF{M}_2$, then $\xi(\FF{M}) = \xi(\FF{M}_1) E \xi(\FF{M}_2) W$.
\item Case II :  the root edge $e$  of $\FF{M}$ is not a disconnecting one, that is $\FF{M} = \maltese_i(\FF{M}_1)$ for some rooted planar map $\FF{M}_1$ and some $0 \leq i \leq c(\FF{M}_1)$, then  we set $\xi(\FF{M}) = \maltese_i(\xi(\FF{M}_1))$.

\end{itemize}

\begin{thm}
	$\xi$ is a bijection between rooted planar maps with $n$ edges, $i+1$ vertices, $j+1$ faces and $k$ corners in the infinite face  and generalized fighting fish of size $n$ with $i$ left lower free edges, $j$ right lower free edges and $k$ non-initial visits at latitude $0$. Moreover, for a rooted planar map $\FF{M}$, $\xi(\FF{M})$ is the Lehman-Lenormand encoding of M, that is the Mullin encoding of the tree-rooted planar map $(\FF{M},T)$ where $T$ is the rightmost depth-first search spanning tree of $\FF{M}$. As a consequence, $\xi$ preserves duality and its restriction to non separable planar maps is $\Phi$.
\end{thm}

\begin{proof}
	Since the two recursive decompositions are isomorphic, $\xi$ is a well-defined bijection between rooted planar maps with $n$ edges and  $k$ corners in the infinite face and generalized fighting fish of size $n$ and $k$ points at latitude $0$. We proceed by induction on the size of rooted planar maps to prove that $\xi$ maps rooted planar maps to their rightmost depth-first search Mullin encoding. Let $\FF{M}$ be a rooted planar map. We distinguish between the two possible cases :
	\begin{itemize}
	\item Case I : the root edge $e$ of $\FF{M}$ is a disconnecting one, then $\xi(\FF{M}) = \xi(\FF{M}_1) E \xi(\FF{M}_2) W$. Then, by induction hypothesis, $\xi(\FF{M}_1)$ and $\xi(\FF{M}_2)$ are the respective Mullin encodings of $\FF{M}_1$ and $\FF{M}_2$ endowed with their respective rightmost DFS spanning tree $T_1$ and $T_2$. The rightmost DFS spanning tree $T$ of $\FF{M}$ is obtained by creating an edge between  the root of $T_1$ and the root of $T_2$, and by performing the counterclockwise tour of $\FF{M}$ around $T$, we get that the Mullin encoding of $(\FF{M},T)$ is exactly  $\xi(\FF{M}_1) E \xi(\FF{M}_2) W$, that is $\xi(\FF{M})$.
          
\item Case II : if $\FF{M} = \maltese_i(\FF{M}_1)$ for some rooted planar map $\FF{M}_1$ and some $0 \leq i \leq c(\FF{M}_1)$, then by the induction hypothesis, $\xi(\FF{M}_1)$ is the Mullin encoding of $\FF{M}_1$ endowed with its rightmost DFS spanning tree $T_1$. Then the rightmost DFS spanning tree of $\FF{M}$ is also $T_1$, because the added edge has one extremity which is the last half edge of the root vertex and it is discovered from its other side in the DFS.. Performing the counterclockwise tour of $\FF{M}$ around $T$, we get that the Mullin encoding of $(\FF{M},T)$ is exactly  $\maltese_i(\xi(\FF{M}_1))$, that is $\xi(\FF{M})$.

 \end{itemize}
\end{proof}

Let us also remark that the bijection $\xi$ preserves the following statistics:
\begin{itemize}
  \item The number of \emph{down bridges} in the generalized fighting fish $\FF{F}$, \emph{i.e.} the number of places where $\FF{F}$ can be decomposed as $\FF{F}=\FF{F}_1 E \FF{G} W \FF{F}_2$ where $\FF{F}_1 \FF{F}_2$ and $\FF{G}$ are (possibly empty) generalized fighting fish, corresponds to the number of bridges in the planar map $\FF{M}$. 
  \item The number of \emph{up bridges} in the generalized fighting fish $\FF{F}$, \emph{i.e.} the number of places where $\FF{F}$ can be decomposed as $\FF{F}=\FF{F}_1 N \FF{G} S \FF{F}_2$ where $\FF{F}_1 \FF{F}_2$ and $\FF{G}$ are (possibly empty) generalized fighting fish, corresponds to the number of loops in the planar map $\FF{M}$. 
\end{itemize}

  \begin{definition}
    \begin{itemize}
    \item A  {\em generalized fighting fish without up bridges}  denoted $\mathcal{UF}$ consists of the generalized fighting fish obtained from the empty path using only operations $\bigtriangledown_k$ with $k \geq 0$ and $\bigtriangleup_k$ with $k \geq 1$.
      \item A {\em generalized fighting fish without down bridges}, denoted $\mathcal{DF}$, consists of the generalized fighting fish obtained from the empty path using only operations $\bigtriangledown_k$ with $k \geq 1$ and $\bigtriangleup_k$ with $k \geq 0$.
      \end{itemize}
    
  \begin{itemize}
       \item A {\em loopless planar map} is a planar
       map without loops, i.e., edges that start and end at the same vertex. We denote by $\mathcal{LM}$ the class of loopless planar maps.
     \item A {\em bridgeless planar map} is a planar
       map without bridges, i.e., edges whose deletion disconnects the map. We denote by $\mathcal{BM}$ the class of bridgeless planar maps.
       \end{itemize}
  \end{definition}
  \begin{figure}[ht]
	\centering
	\includegraphics[scale=.8]{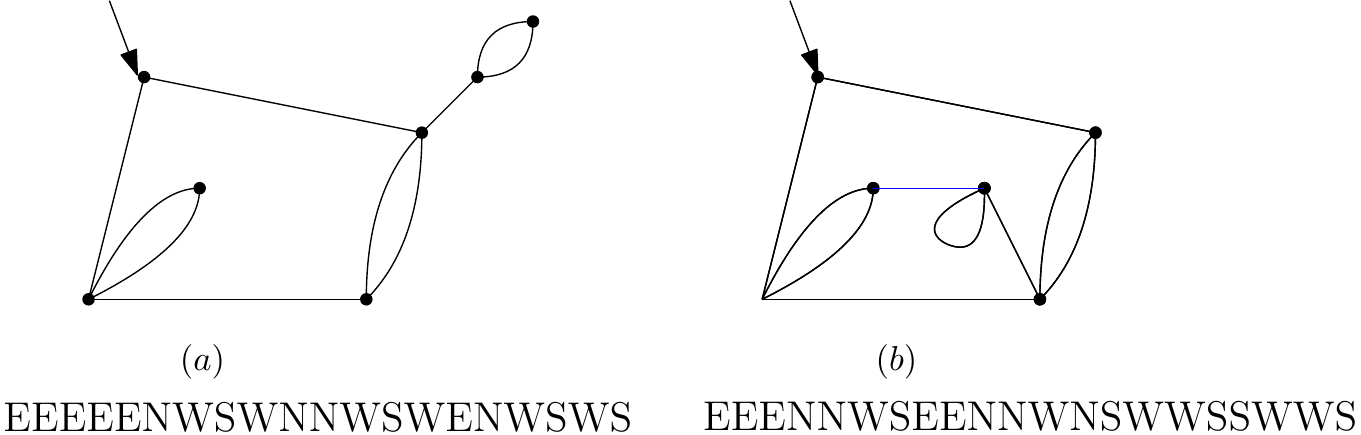}
	\caption{$(a)$ A loopless planar map and its associated generalized fighting fish without up bridges; $(b)$ A bridgeless planar map and its associated fighting fish without down bridges}
	\label{fig:BridgeLoopLess}
\end{figure}

  In Figure~\ref{fig:BridgeLoopLess} there is an example of a loopless and a bridgeless planar maps and their corresponding generalized fighting fish (in terms of words) by the bijection $\xi$.

\begin{prop}
  The previous bijection specializes into bijections between $\mathcal{UF}_n$ and $\mathcal{LM}_n$, between $\mathcal{BF}_n$ and $\mathcal{BM}_n$, and between $\mathcal{FF}_n$ and $\mathcal{NS}_n$.

\end{prop}
We point out that the length of the jaw of the generalized fighting fish, that is the length of the first sequence of $E$ steps, gives the number of edges of the infinity face minus $1$. An other nice property of the bijection is that the notion of duality in maps is translated into the natural notion of symmetry for generalized fighting fishes : exchange the operations $\bigtriangleup_k$ and $\bigtriangledown_k$ during the construction.

\begin{figure}[H]
	\centering
	\includegraphics[page=1,scale=0.6]{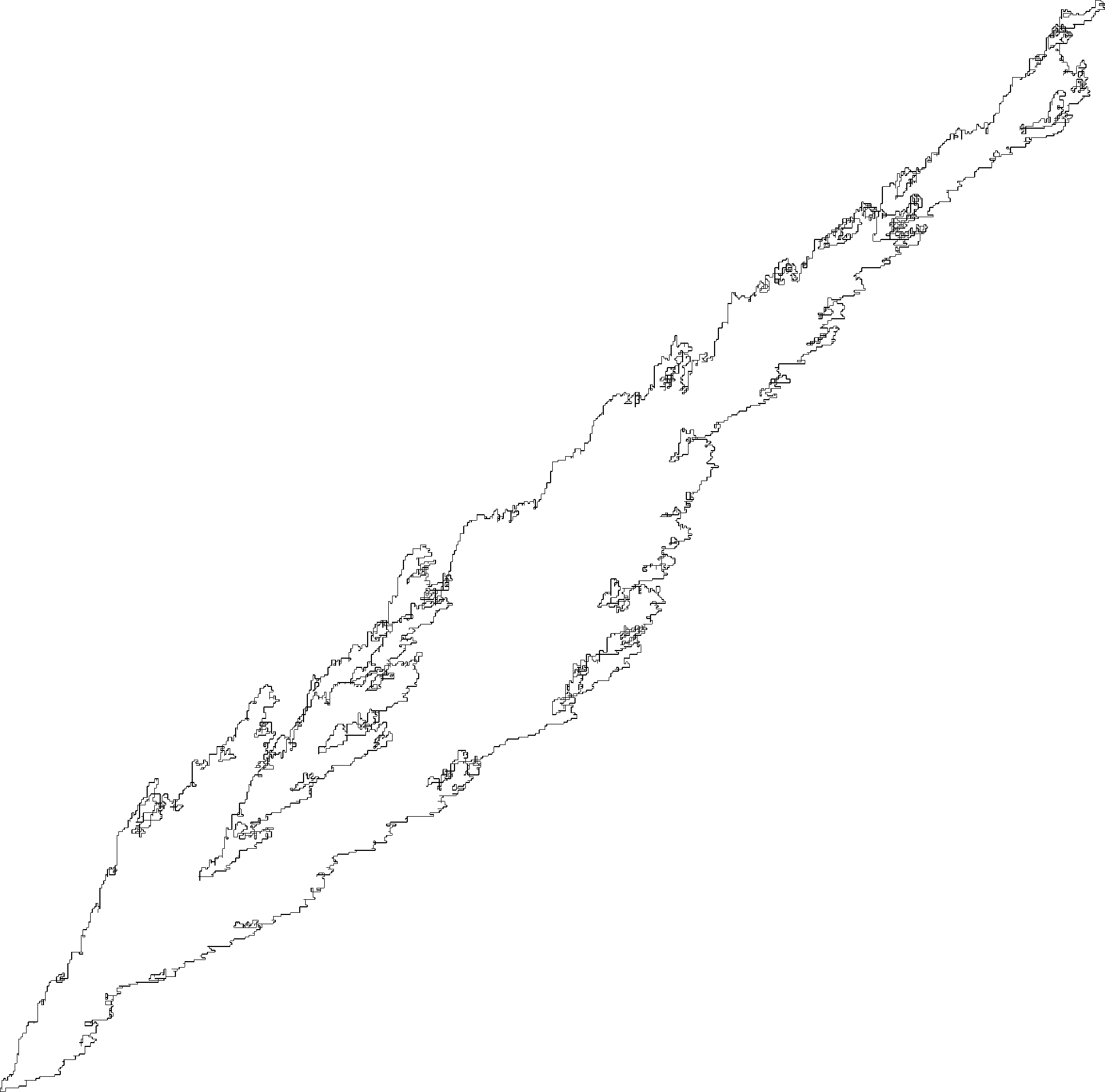}
	\caption{A uniform random fighting fish of size 10000.}
	\label{fig:BigFish}
\end{figure}

\section{Conclusion and perspectives}
We have shown that Mullin encoding gives, in the special case of Lehman-Lenormand code, a direct bijection between rooted planar maps and generalized fighting fish, that further specializes in a bijection between rooted planar maps and fighting fish. We observe that the Lehman-Lenormand code being based on a single DFS, it can be performed in linear time: therefore the known linear time Schaeffer's random generator~\cite{Sc99} for nonseparable planar maps and rooted planar maps (available at \texttt{https://www.lix.polytechnique.fr/Labo/Gilles.Schaeffer/PagesWeb/PlanarMap/}) immediately gives linear random generators for fighting fish and generalized fighting fish. In Figure~\ref{fig:BigFish} we show a random fighting fish of size $10000$ obtained by adapting this generator.

While Lehman-Lenormand code dates back to the early seventies, the main contribution of this article is to enlighten that the counterclockwise codes of (nonseparable) rooted planar maps endowed with their rightmost DFS spanning trees can be given a very natural geometric interpretation as a branching surface obtained by gluing unit cells. This geometric point of view based on fighting fish already helped in proving a simple formula for the distance of intervals in the Tamari lattice, using the correspondence of the area statistic on fighting fish, where the area of a  (generalized) fighting fish is the number of cells it contains. This raises the problem of finding a natural definition of the corresponding (by $\xi$) statistic on rooted planar maps. \\
Another perspective would be to unify the two extensions of fighting fish we presented in this paper and th preceeding one. Recall that the class of extended fighting fish is obtained by allowing also the operation of replacing an occurrence of the substring $WN$ by a vertical step $V$ in the definition of fighting fish. It would be nice to have a general model of fish embedding the two classes of generalized and extended fighting fish, and that would fit nicely with the bijections we defined, either with Tamari intervals and rooted planar maps.
\bibliographystyle{plain}
\bibliography{fpsac022}

\end{document}